\documentclass[twoside]{article}
\usepackage{amsmath,amsthm,amssymb,bm}
\usepackage{graphicx}
\usepackage[reset,a4paper,vmargin=3truecm,hmargin=2truecm]{geometry}

\newcommand{\set}[2]{\left\{#1\,\middle\vert\,#2\right\}}
\newcommand{\card}[1]{\##1}
\newcommand{\abs}[1]{\left|#1\right|}
\newcommand{\ord}{\phi}
\newcommand{\ords}{\ord_{\text{\upshape st}}}
\newcommand{\ordh}{\ord_{\text{\upshape hom}}}
\newcommand{\spec}{f}
\newcommand{\specp}{\spec_{\text{\upshape pr}}}
\newcommand{\speco}{\spec_{\text{\upshape ord}}}
\newcommand{\specC}{\spec_{\text{\upshape Cay}}}
\newcommand{\C}{\mathbb{C}}
\newcommand{\Z}{\mathbb{Z}}
\newcommand{\1}{\bm{1}}
\newcommand{\adet}[1][\alpha]{\operatorname{det}_{#1}}
\DeclareMathOperator{\wrdet}{wrdet}
\DeclareMathOperator{\diag}{diag}
\DeclareMathOperator{\sgn}{sgn}
\DeclareMathOperator{\per}{per}
\DeclareMathOperator{\AT}{AT}
\newcommand{\floor}[1]{\left\lfloor#1\right\rfloor}
\newcommand{\gen}[1]{\left\langle#1\right\rangle}
\newcommand{\sym}[1]{\mathfrak{S}_{#1}}
\newcommand{\ind}[2]{[#1:#2]}
\newcommand{\cy}[1]{\Z_{#1}} 
\newcommand{\wD}{\Theta}
\newcommand{\wM}{X}
\newcommand{\e}{\varepsilon}

\newtheorem{thm}{Theorem}
\newtheorem{lem}[thm]{Lemma}
\newtheorem{prop}[thm]{Proposition}

\theoremstyle{definition}
\newtheorem{ex}{Example}

\theoremstyle{remark}
\newtheorem{rem}{Remark}

\newcommand{\MSC}[2]{\smallskip\noindent
{\bfseries 2010 Mathematics Subject Classification:} {\itshape Primary} #1, {\itshape Secondary} #2.}
\newcommand{\Keywords}[1]{\smallskip\noindent
{\bfseries Keywords and phrases:} #1.}

\newcommand{\Grand}[2]{\thanks{Partially supported by Grand-in-Aid for  #1 No. #2 of JSPS.}}

\title{\sffamily Wreath determinants for group-subgroup pairs}
\author{Kei HAMAMOTO, Kazufumi KIMOTO\Grand{Scientific Research (C)}{25400044},\\
Kazutoshi TACHIBANA and Masato WAKAYAMA\Grand{Challenging Exploratory Research}{25610006}}
\date{June 10, 2014}
\begin{document}
\maketitle

\begin{abstract}
The aim of the present paper is to generalize the notion of the group determinants for finite groups.
For a finite group $G$ and its subgroup $H$, one may define a rectangular matrix
of size $|H|\times |G|$ by
$X=\bigl(x_{hg^{-1}}\bigr)_{{h\in H,g\in G}}$,
where $\set{x_g}{g\in G}$ are indeterminates indexed by the elements in $G$.
Then, we define an invariant $\Theta(G,H)$ for a given pair $(G,H)$
by the $k$-wreath determinant of the matrix $X$,  where $k$ is the index of $H$ in $G$.
The $k$-wreath determinant of $n$ by $kn$ matrix is a relative invariant of the left action by the general linear group of order $k$ and right action by
the wreath product of two symmetric groups of order $k$ and $n$.
Since the definition of $\Theta(G,H)$ is \emph{ordering-sensitive}, representation theory of symmetric groups are
naturally involved. In this paper, we treat abelian groups with a special choice of indeterminates  and give various
examples of non-abelian group-subgroup pairs.

\MSC{20C15}{20C30, 20E22, 05E15}

\Keywords{group determinants, wreath determinants, finite groups, symmetric groups, characters, Cayley graphs}
\end{abstract}

\section{Introduction}

It is Frobenius who initiated the character theory of finite groups \cite{C1999}.
At the very first stage of his study,
the \emph{group determinant} $\Theta(G)$ of a given finite group $G$ defined as the determinant of the
group matrix $\bigl(x_{uv^{-1}}\bigr)_{u,v\in G}$
\begin{equation}
\Theta(G):=\det\bigl(x_{uv^{-1}}\bigr)_{u,v\in G},
\end{equation}
where $\set{x_g}{g\in G}$ are indeterminates indexed by the elements in $G$,
played an important role.
(One should note that the definition of $\Theta(G)$ is independent of the choice of the ordering of elements in $G$.)
Indeed, the group determinant $\Theta(G)$ reflects the structure of the regular representation of $G$,
which contains all the irreducible representation of $G$.
The factorization of $\Theta(G)$ is corresponding to the
irreducible decomposition of the regular representation,
and the irreducible character values appear as coefficients in the factors.
In 1991, Formanek and Sibley \cite{FS1991PAMS} shows that
two groups are isomorphic if and only if their group determinants coincide
under a suitable correspondence between the sets of indeterminates for these groups:
\begin{equation}
\Theta(G)=\Theta(G') \iff G\cong G'.
\end{equation}
Namely, the group determinant is a perfect invariant for finite groups.

In this paper,  we extend the notion of group determinants. Actually, we
define an invariant $\Theta(G,H)$ for a given pair $(G,H)$,
$G$ being a finite group and $H$  its subgroup, by employing the \emph{wreath determinant} \cite{KW2008}.
For a positive integer $k$, the 
\emph{$k$-wreath determinant} $\wrdet_k$  is a polynomial function
on the set of $n$ by $kn$ matrices for each positive integer $n$
characterized by
(i) multilinearity in column vectors,
(ii) relative $GL_n$-invariance from the left,
and
(iii) $\sym k^n$-invariance with respect to permutations in columns
(see Section \ref{sec:adet_and_wrdet} for the precise definition).
Roughly, $\Theta(G,H)$ is defined to be
$$
\Theta(G,H):=\wrdet_k\bigl(x_{hg^{-1}}\bigr)_{\substack{h\in H\\ g\in G}},
$$
where $k=\ind GH$ is the index of $H$ in $G$.
In fact, since $\wrdet_k$ is \emph{not} relative invariant under general permutations in columns (i.e. the action
of $\sym{kn}$ from the right),
we should take account of the \emph{ordering} of $G$ to define $\Theta(G,H)$.
This is a crucial difference from $\Theta(G)$.
We note that $\Theta(G,G)$ is nothing but the original group determinant $\Theta(G)$ since
the $1$-wreath determinant is the ordinary determinant.

It would be fundamental and natural to explore an analog of the Frobenius character theory
as well as e.g. Formanek-Sibley type theorems for $\Theta(G,H)$. There are, however, certain obstacles in the study.
One of the most essential one is the fact that
the definition of $\Theta(G,H)$ is \emph{ordering-sensitive};
If we exchange the ordering of the columns in the matrix $\bigl(x_{hg^{-1}}\bigr)_{{h\in H,g\in G}}$ (sometimes called
as a group-subgroup matrix),
then its wreath determinant becomes rather different from the one before manipulated.
Actually, one needs to taking account of representations of symmetric groups of order $kn$ and $k$.
Therefore, as a small first step,
we analyze $\Theta(G,H)$ when $G$ is a finite \emph{abelian} group
under a certain \emph{specialization of indeterminates},
in which case the difficulties mentioned above are fairly reduced.
We give a factorization of $\Theta(G,H)$ when $H$ is a direct product of several components in $G=\Z/m_1\Z\times \cdots\times \Z/m_l\Z$
and the indeterminates are specialized to powers of another indeterminate $q$
according to a suitably chosen ordering of elements in $G$. 
Imitating the result and proof for finite abelian groups, we give also certain  
computations for non-abelian group-subgroup pairs under some particular condition. 
In the last section, 
we will give several examples for non-abelian groups as well as another specializations
of indeterminates. The examples include what we call a Cayley specialization,
which is intimately related to the graph theory \cite{Cid}.



%
%
%

\section{Wreath determinants for group-subgroup pairs}

\subsection{Alpha-determinants and wreath determinants}\label{sec:adet_and_wrdet}

Let $\alpha$ be a complex parameter.
The \emph{alpha-determinant} of a square matrix $X=(x_{ij})\in M_n$ is defined by
\begin{equation}
\adet X=\sum_{\sigma\in\sym n}\alpha^{\nu(\sigma)}x_{\sigma(1)1}\dots x_{\sigma(n)n},
\end{equation}
where $\nu(\sigma)=\sum_{j\ge2}(j-1)c_j(\sigma)$, $c_j(\sigma)$ being the number of
$j$-cycles in $\sigma$. (The notion of alpha determinants was first introduced by Vere-Jones  \cite{VJ1988}
as  ``the $\alpha$-permanent''.
By \cite{ST2003}, it was renamed as the alpha determinant.)

For an $n$ by $kn$ matrix $X=(x_{ij})\in M_{n,kn}$,
the \emph{$k$-wreath determinant} of $X$ is defined by
\begin{equation}
\wrdet_k X=\adet[-1/k](X\otimes\1_{k,1}),
\end{equation}
where $A\otimes B$ denotes the Kronecker product of $A$ and $B$,
$\1_{k,1}$ is the $k$ by $1$ all-one matrix \cite{KW2008}. We note here
that if we look at the irreducible decomposition of the
$GL_n$-cyclic module generated by $\adet X$, the distinguished 
phenomena happens when $\alpha=-1/k$ $(k=1,2,...,n-1)$ so that
$\adet[-1/k]$ shares some basic property of determinants \cite{MW2006}. This can be seen from the fact that each number $\pm1/k$ is a root of content polynomials \cite{Mac}. Remark that $\adet[-1]=\det$ and $\adet[1]=\per$, the permanent.

\begin{ex}
The $2$-wreath determinant of
\begin{equation*}
A=\begin{pmatrix}
a_1 & a_2 & a_3 & a_4 \\
b_1 & b_2 & b_3 & b_4
\end{pmatrix}
\end{equation*}
is given by
\begin{align*}
\wrdet_2A&=
\adet[-1/2]\begin{pmatrix}
a_1 & a_2 & a_3 & a_4 \\
a_1 & a_2 & a_3 & a_4 \\
b_1 & b_2 & b_3 & b_4 \\
b_1 & b_2 & b_3 & b_4
\end{pmatrix} \\
&=\frac14(a_1a_2b_3b_4+b_1b_2a_3a_4)
-\frac18(
a_1b_2a_3b_4+a_1b_2b_3a_4+b_1a_2a_3b_4+b_1a_2b_3a_4)
\\
&=\frac18\begin{vmatrix}
a_1 & a_3 \\ b_1 & b_3
\end{vmatrix}\begin{vmatrix}
a_2 & a_4 \\ b_2 & b_4
\end{vmatrix}
+\frac18\begin{vmatrix}
a_1 & a_4 \\ b_1 & b_4
\end{vmatrix}\begin{vmatrix}
a_2 & a_3 \\ b_2 & b_3
\end{vmatrix}.
\end{align*}
\end{ex}

The following result is fundamental (see \cite{KW2008} for the proof and \cite{K2014}).
\begin{prop}
Let $k,n$ be positive integers.
Put $f(X)=\wrdet_k X$ for $X\in M_{n,kn}$.
Then $f$ is a map from $M_{n,kn}$ to $\C$ satisfying the following conditions:
\begin{enumerate}
\item
$f$ is multilinear with respect to columns.
\item
$f(gX)=(\det g)^kf(X)$ for any $g\in GL_n$.
\item
$f(XP(\tau))=f(X)$ for any $\tau\in\sym k^n$,
where $P(\tau)=(\delta_{i\tau(j)})$ is the permutation matrix for $\tau$.
\end{enumerate}
Conversely, if a map $f\colon M_{n,kn}\to\C$ satisfies the three conditions above,
then $f$ equals the $k$-wreath determinant up to constant multiple.
\qed
\end{prop}

\begin{rem}
The right $\sym k^n$-invariance of the $k$-wreath determinant above
extends to the relative right $\sym k\wr\sym n$-invariance
\begin{equation*}
\wrdet_k XP(g)=(\sgn\sigma)^k\wrdet_k X,\qquad
g=(\tau,\sigma)\in\sym k\wr\sym n,
\end{equation*}
where $\sym k\wr\sym n=\sym k^n\rtimes\sym n$ is the wreath product of $\sym k$ and $\sym n$,
which we regard as a subgroup of $\sym{kn}$.
\end{rem}

\subsection{Wreath determinants associated with a pair $(G,H)$}

Let $G$ be a finite group of order $m=kn$, and $H$ be a subgroup of $G$ of order $n$.
Suppose that a bijection $\ord\colon\{0,1,\dots,m-1\}\to G$ called an \emph{ordering of $G$} is given.
We put $g_i:=\ord(i)$ for short.

Let $R$ be a commutative ring, and $\spec\colon G\to R$ be a map called a \emph{specialization}.
We sometimes write $\spec(g)=x_g$ or $\spec(\ord(i))=\spec(g_i)=x_i$ for short.
Define
\begin{equation}
\wD(G,H,\ord,\spec):=\wrdet_k \wM(G,H,\ord,\spec),\qquad
\wM(G,H,\ord,\spec):=\Bigl(\spec(h_ig_j^{-1})\Bigr)_{\substack{0\le i< n\\ 0\le j< m}},
\end{equation}
where $\wrdet_k$ denotes the $k$-wreath determinant
and $H=\{h_0,\dots,h_{n-1}\}$.
If the ordering $\ord$ and the specialization $\spec$ are clear in the context,
then we omit them and write simply $\wD(G,H)$.

\begin{ex}
$\wD(G,\spec):=\wD(G,G,\ord,\spec)=\det(\spec(g_ig_j^{-1}))$ is the ordinary group determinant.
In this case, the ordering $\ord$ is irrelevant.
\end{ex}

\begin{ex}
We have $\wD(G,\{e\},\ord,\spec)=\frac{k!}{k^k}\prod_{g\in G}\spec(g)$.
If $\spec(g)=x_{o(g)}$ for $g\in G$, where $o(g)$ is the order of $g$, then
$\wD(G,\{e\},\ord,\spec)=\frac{k!}{k^k}\prod_{i\ge1}x_i^{\card{\set{g\in G}{o(g)=i}}}$
tells us the distribution of orders of elements in $G$.
\end{ex}

\begin{ex}
Let $G=\{g_0=e,g_1=a,g_2=a^2,g_3=a^3\}$ be the cyclic group of order $4$ with the `standard' ordering,
and take $H=\{h_0=g_0=e,h_1=g_2=a^2\}$. We have
\begin{equation*}
\wM(G,H)
=\begin{pmatrix}
\spec(h_0g_0^{-1}) & \spec(h_0g_1^{-1}) & \spec(h_0g_2^{-1}) & \spec(h_0g_3^{-1}) \\
\spec(h_1g_0^{-1}) & \spec(h_1g_1^{-1}) & \spec(h_1g_2^{-1}) & \spec(h_1g_3^{-1})
\end{pmatrix}
=\begin{pmatrix}
\spec(e) & \spec(a^3) & \spec(a^2) & \spec(a) \\
\spec(a^2) & \spec(a) & \spec(e) & \spec(a^3)
\end{pmatrix}
\end{equation*}
and
\begin{gather*}
\wD(G,H)=\frac18\begin{vmatrix}
\spec(e) & \spec(a^2) \\
\spec(a^2) & \spec(e)
\end{vmatrix}
\begin{vmatrix}
\spec(a^3) & \spec(a) \\
\spec(a) & \spec(a^3)
\end{vmatrix}+\frac18\begin{vmatrix}
\spec(e) & \spec(a) \\
\spec(a^2) & \spec(a^3)
\end{vmatrix}
\begin{vmatrix}
\spec(a^3) & \spec(a^2) \\
\spec(a) & \spec(e)
\end{vmatrix}\\
=\frac18(\spec(e)^2-\spec(a^2)^2)(\spec(a^3)^2-\spec(a)^2)+\frac18(\spec(e)\spec(a^3)-\spec(a)\spec(a^2))^2.
\end{gather*}
If we assume that $\spec(g_i)=\spec(a^{i})=q^{i}\in\C[q]$, then we have
\begin{equation*}
\wD(G,H)=-\frac18q^2(1-q^4)^2.
\end{equation*}
\end{ex}

\begin{ex}\label{ex:Klein 4-group}
Let $G=\{g_0=e, g_1=a, g_2=b, g_3=ab\}$ be the Klein four-group
(i.e. $a^2=b^2=e$, $ab=ba$),
and take subgroups $H=\{e,a\}$, $H'=\{e,b\}$ and $H''=\{e,ab\}$ of order 2.
We have
\begin{align*}
\wD(G,H)
&=\wrdet_2\begin{pmatrix}
\spec(e) & \spec(a) & \spec(b) & \spec(ab) \\
\spec(a) & \spec(e) & \spec(ab) & \spec(b)
\end{pmatrix}\\
&=-\frac18(\spec(e)\spec(ab)-\spec(a)\spec(b))^2
-\frac18(\spec(e)\spec(b)-\spec(a)\spec(ab))^2,
\end{align*}
\begin{align*}
\wD(G,H')
&=\wrdet_2\begin{pmatrix}
\spec(e) & \spec(a) & \spec(b) & \spec(ab) \\
\spec(b) & \spec(ab) & \spec(e) & \spec(a)
\end{pmatrix}\\
&=\frac18(\spec(e)^2-\spec(b)^2)(\spec(a)^2-\spec(ab)^2)
+\frac18(\spec(e)\spec(a)-\spec(b)\spec(ab))^2,
\end{align*}
\begin{align*}
\wD(G,H'')
&=\wrdet_2\begin{pmatrix}
\spec(e) & \spec(a) & \spec(b) & \spec(ab) \\
\spec(ab) & \spec(b) & \spec(a) & \spec(e)
\end{pmatrix}\\
&=\frac18(\spec(e)\spec(a)-\spec(b)\spec(ab))^2
+\frac18(\spec(e)^2-\spec(ab)^2)(\spec(a)^2-\spec(b)^2).
\end{align*}
If we assume that $\spec(g_i)=q^{i}\in\C[q]$, then we have
\begin{equation*}
\wD(G,H)=-\frac18q^4(1-q^2)^2,\qquad
\wD(G,H')=\frac14q^2(1-q^4)^2,\qquad
\wD(G,H'')=\frac18q^2(1-q^2)^2(2+3q^2+2q^4),
\end{equation*}
which are different from each other.
Thus $H\cong H'$ does not imply $\wD(G,H)=\wD(G,H')$ in general.
\end{ex}

\begin{rem}
Denote by $[H]$ the set of all isomorphism classes of $H$ in $G$. Then, as the example above shows,
the map $[H] \to \wD(G,H)$
is a multivalued function. The precise/deep understanding of this fact would be important.
For instance, does the collection $\{\wD(G,H)\}_{[H]}$ determine an isomorphism class of the
pair $(G,H)$?
\end{rem}

\section{Finite abelian group-subgroup pair case}
\subsection{Standard ordering and principal specialization}

Let $\cy{r}=\Z/r\Z=\{0,1,2,\dots,r-1\}$ be the cyclic group of order $r$,
where we write $j$ to indicate $j+r\Z$ for simplicity.

Assume that
$G=\cy{m_1}\times\dots\times\cy{m_l}$.
Put
\begin{equation}
M_j=\prod_{i=1}^{j-1}m_i\qquad(j=1,2,\dots,l),\qquad
m=m_1m_2\dots m_l=\card G,
\end{equation}
and fix the ordering $\ords$ by
\begin{equation}
g_{i}=\ords(i)
=\Bigl(\floor{i/{M_1}}\bmod m_1,\dots,\floor{i/{M_l}}\bmod m_l\Bigr)
\qquad(i=0,1,\dots,m-1),
\end{equation}
where $\floor x$ denotes the largest integer which is not greater than $x$
and $a\bmod b$ denotes the remainder of $a$ 
divided by $b$. We call  $\ords$ the \emph{standard ordering}.
We also take a specialization $\specp\colon G\to \C[q]$ as $\specp(g_i)=q^{i}$,
which we call the \emph{principal specialization}.

\begin{ex}
When $G=\cy3\times\cy2\times\cy2$, we have
\begin{gather*}
g_0=(0,0,0),\quad g_1=(1,0,0),\quad g_2=(2,0,0),\quad
g_3=(0,1,0),\quad g_4=(1,1,0),\quad g_5=(2,1,0),\\
g_8=(0,0,1),\quad g_7=(1,0,1),\quad g_8=(2,0,1),\quad
g_{9}=(0,1,1),\quad g_{10}=(1,1,1),\quad g_{11}=(2,1,1).
\end{gather*}
\end{ex}

\subsection{Result}

Let $m_1,m_2,\dots,m_l$ and $n_1,n_2,\dots,n_l$ be positive integers such that
$n_s\mid m_s$ for each $s$. We put $k_s=m_s/n_s$ and
\begin{gather*}
M_s=\prod_{i<s}m_i\quad(s=1,2,\dots,l),\qquad
N_s=\prod_{i<s}n_i\quad(s=1,2,\dots,l),\\
m=\prod_{s=1}^l m_s,\qquad
n=\prod_{s=1}^l n_s,\qquad
k=\prod_{s=1}^l k_s.
\end{gather*}

Let $G=\cy{m_1}\times\cy{m_2}\times\dotsb\times\cy{m_l}$.
We take a subgroup
$H=\cy{n_1}\times\cy{n_2}\times\dotsb\times\cy{n_l}$,
where we regard $\cy{n_s}$ as the subgroup of $\cy{m_s}$ generated by $k_s$:
$\cy{m_s}=\{0,k_s,2k_s,\dots,(m_s-1)k_s\}$.
Notice that $\card G=m$, $\card H=n$ and $\ind GH=k$.
In this case, we have
\begin{equation}\label{eq:matrix in group wreath determinant}
\wM(G,H,\ords,\specp)=\Bigl(q^{\e_l(i,j)}\Bigr)_{\substack{0\le i<n\\ 0\le j<m}},\qquad
\e_l(i,j)=\sum_{s=1}^l {M_s((k_s\floor{i/N_s}-\floor{j/M_s})\bmod m_s)}.
\end{equation}

Then the following factorization of the wreath determinant for a finite abelian group-subgroup pair holds. 
\begin{thm}\label{MainTheorem}
Retain the assumption and notation above. 
Then one has 
\begin{equation}\label{eq:general finite abelian case}
\wD(G,H,\ords,\specp)=\omega^{(k^n)}(\sigma\tau^{-1})
\Bigl(\frac{k!}{k^k}\Bigr)^n
\prod_{s=1}^l q^{mM_s(k_s-1)/2}
\prod_{s=1}^l (q^{M_sm_s}-1)^{m(1-1/n_s)},
\end{equation}
where $\sigma$ and $\tau$ are permutations of $m$ letters determined by the conditions
\begin{gather*}
(I_{n_l}\otimes\1_{1,k_l})\otimes\dotsb\otimes(I_{n_1}\otimes\1_{1,k_1})=(I_n\otimes\1_{1,k})P(\sigma), \\
P((1~2~\dots~m_l))\otimes\dotsb\otimes P((1~2~\dots~m_1))=P(\tau).
\end{gather*}
The function $\omega^{(k^n)}$ on $\sym m$ is defined by
\begin{equation*}
\omega^{(k^n)}(x)=\frac1{(k!)^n}\sum_{g\in\sym k^n}\chi^{(k^n)}(xg)\qquad(x\in\sym m),
\end{equation*}
where $\chi^{(k^n)}$ is the irreducible character of $\sym m$
corresponding to the partition $(k^n)=(k,\dots,k)\vdash m$.
\end{thm}

The proof of the theorem will be given in the subsequent subsection. 

\begin{ex}[Cyclic group case]\label{ex:cyclic group case}
If $l=1$, then \eqref{eq:general finite abelian case} reads
\begin{equation*}
\wD(\cy{m},\cy n,\ords,\specp)=\omega^{(k^n)}(\tau^{-1})
\Bigl(\frac{k!}{k^k}\Bigr)^n
q^{m(k-1)/2}
(q^{m}-1)^{k(n-1)},\qquad
\tau=(1~2~\dots~m).
\end{equation*}
By Remark 5.5 in \cite{K2014}, we have
\begin{equation*}
\omega^{(k^n)}(\tau^{-1})=\omega^{(k^n)}(\tau)
=\frac{\text{the coefficient of $(x_{11}x_{22}\dots x_{nn})^{k-1}x_{12}x_{23}\dots x_{n1}$ in $(\det X)^k$}}
{\left|\sym k^n:\sym k^n\cap\tau^{-1}\sym k^n\tau\right|}.
\end{equation*}
It is elementary to see that
\begin{equation*}
\text{the coefficient of $(x_{11}x_{22}\dots x_{nn})^{k-1}x_{12}x_{23}\dots x_{n1}$ in $(\det X)^k$}=(-1)^{n-1}k
\end{equation*}
and
\begin{equation*}
\left|\sym k^n:\sym k^n\cap\tau^{-1}\sym k^n\tau\right|=\frac{k!^n}{(k-1)!^n}=k^n.
\end{equation*}
Thus we have
\begin{equation*}
\omega^{(k^n)}(\tau)=\Bigl(-\frac1k\Bigr)^{n-1}.
\end{equation*}
Hence it follows that
\begin{equation}
\wD(\cy{m},\cy n,\ords,\specp)
=\Bigl(-\frac1k\Bigr)^{n-1}\Bigl(\frac{k!}{k^k}\Bigr)^n
q^{m(k-1)/2}(q^{m}-1)^{k(n-1)}.
\end{equation}
\end{ex}

\begin{ex}
If $n_s=m_s$ for $s=1,2,\dots,r$ and $n_s=1$ for $s=r+1,\dots,l$,
then \eqref{eq:general finite abelian case} reads
\begin{align*}
&\wD(\cy{m_1}\times\dotsb\times\cy{m_l},\cy{m_1}\times\dotsb\times\cy{m_r},\ords,\specp) \\
&=\omega^{(k^n)}(\sigma\tau^{-1})
\Bigl(\frac{k!}{k^k}\Bigr)^n
\prod_{s=r+1}^l q^{mM_s(m_s-1)/2}
\prod_{s=1}^r (q^{M_sm_s}-1)^{m(1-1/n_s)} \\
&=\omega^{(k^n)}(\sigma\tau^{-1})
\Bigl(\frac{k!}{k^k}\Bigr)^n
q^{n^2k(k-1)/2}
\left\{\prod_{s=1}^r (q^{m_1m_2\dots m_s}-1)^{n(1-1/n_s)}\right\}^k.
\end{align*}
\end{ex}

\begin{ex}
Let $G=\cy n\times\cy n$ and consider the subgroups
$H=\cy n\times\cy 1$, $H'=\cy1\times\cy n$ of $G$.
We have
\begin{align*}
\wD(G,H,\ords,\specp)&=
\omega^{(n^n)}(\sigma\tau^{-1})
\Bigl(\frac{n!}{n^n}\Bigr)^n
q^{n^3(n-1)/2}(q^n-1)^{n(n-1)},\\
\wD(G,H',\ords,\specp)&=
\omega^{(n^n)}(\tau^{-1})
\Bigl(\frac{n!}{n^n}\Bigr)^n
q^{n^2(n-1)/2}(q^{n^2}-1)^{n(n-1)},
\end{align*}
where the permutations $\sigma,\tau\in\sym{n^2}$ are given by the conditions
\begin{equation*}
P(\tau)=P((1~2~\dots~n))\otimes P((1~2~\dots~n)),\qquad
\1_{1,n}\otimes I_n=(I_n\otimes\1_{1,n})P(\sigma).
\end{equation*}
By a similar calculation given in Example \ref{ex:cyclic group case},
we see that $\omega^{(n^n)}(\tau^{-1})=1$ and
\begin{equation*}
\omega^{(n^n)}(\sigma\tau^{-1})
=\frac{\AT(n)}{n!^n},\qquad
\AT(n)=\text{the coefficient of $\prod_{i,j=1}^nx_{ij}$ in $(\det X)^n$}.
\end{equation*}
Hence we have
\begin{align*}
\wD(G,H,\ords,\specp)&=
\frac{\AT(n)}{n^{n^2}}
q^{n^3(n-1)/2}(q^n-1)^{n(n-1)},\\
\wD(G,H',\ords,\specp)&=
\Bigl(\frac{n!}{n^n}\Bigr)^n
q^{n^2(n-1)/2}(q^{n^2}-1)^{n(n-1)}.
\end{align*}
When $n=2$, for instance, we have $\sigma=(2~3)$, $\tau=(1~4)(2~3)$, and
\begin{equation*}
\omega^{(2^2)}(\sigma\tau^{-1})=-\frac12,\qquad
\omega^{(2^2)}(\tau^{-1})=1.
\end{equation*}
This partially recovers Example \ref{ex:Klein 4-group}.
In the case where $H''=\Delta\cy n:=\set{(x,x)}{x\in\cy n}$,
the wreath determinant $\wD(G,H'',\ords,\specp)$ would not have simple expression.
\end{ex}

\begin{rem}
The number $\abs{\AT(n)}$ is equal to the difference of
the numbers of even and odd \emph{Latin squares} of size $n$.
It is conjectured that $\AT(n)\ne0$ if $n$ is even (Alon-Tarsi Conjecture \cite{AT1992}).
It is easy to see that $\AT(n)=0$ if $n$ is odd and $n\ge3$.
\end{rem}

%


\subsection{Proof of the theorem}

Put
\begin{equation}
T(m,n;x):=\bigl(x^{(ki-j)\bmod m}\bigr)_{\substack{0\le i<n\\ 0\le j<m}}
\qquad(m=kn)
\end{equation}
and $T(m;x):=T(m,m;x)$.
Notice that
$\wD(\cy{m},\cy{n},\ords,\specp)=\wrdet_k T(m,n;q)$.
It is elementary to see the
\begin{lem}
It holds that
$\det T(m;x)=(1-x^m)^{m-1}$.
\qed
\end{lem}

We notice the following elementary fact on the Kronecker product of two matrices:
If
\begin{equation*}
A=\bigl(a(i,j)\bigr)_{\substack{0\le i<m\\ 0\le j<n}},\qquad
B=\bigl(b(i,j)\bigr)_{\substack{0\le i<p\\ 0\le j<s}},
\end{equation*}
then we have
\begin{equation}\label{eq:explicit Kronecker product}
A\otimes B=\Bigl(
a(\floor{i/p},\floor{j/s})b(i\bmod p,j\bmod s)
\Bigr)_{\substack{0\le i<mp\\ 0\le j<ns}}.
\end{equation}

\begin{lem}
It holds that
\begin{equation}
T(m,n;x)=P(\sigma)\cdot T(n;x^k)\otimes\1_{1,k}\cdot I_n\otimes\diag(x^{k-1},\dots,x,1)\cdot P(\tau)^{-1},
\end{equation}
where $m=kn$, $\sigma=(1~2~\dots~n)\in\sym n$ and $\tau=(1~2~\dots~m)\in\sym m$.
\end{lem}

\begin{proof}
The $(i,j)$-entry of $P(\sigma)^{-1}T(m,n;x)P(\tau)$ is
\begin{equation*}
x^{(k\sigma(i)-\tau(j))\bmod m}=x^{(k(i+1)-(j+1))\bmod m}.
\end{equation*}
On the other hand, the $(i,j)$-entry of $T(n;x^k)\otimes\1_{1,k}\cdot I_n\otimes\diag(x^{k-1},\dots,x,1)$ is
given by 
\begin{equation*}
x^{e(i,j)},\qquad
e(i,j)=k\bigl\{ (i-\floor{j/k}) \bmod n\bigr\}+k-1-(j\bmod k).
\end{equation*}
Since
\begin{equation*}
0\le e(i,j)<m,\qquad
k\bigl\{ (i-\floor{j/k}) \bmod n\bigr\}=(ki-k\floor{j/k})\bmod m,\qquad
j=k\floor{j/k}+(j\bmod k),
\end{equation*}
we have
\begin{equation*}
e(i,j)=e(i,j)\bmod m=\Bigl( ki-k\floor{j/k}+k-1-(j\bmod k)\Bigr)\bmod m=\Bigl( k(i+1)-(j+1) \Bigr)\bmod m
\end{equation*}
as desired.
%
%
%
%
\end{proof}

\begin{lem}
It holds that
\begin{equation*}
\wM(G,H,\ords,\specp)=T(m_l,n_l;q^{M_l})\otimes\dotsb
\otimes T(m_2,n_2;q^{M_2})\otimes T(m_1,n_1;q^{M_1}).
\end{equation*}
\end{lem}

\begin{proof}
The assertion is trivial when $l=1$.
In view of \eqref{eq:matrix in group wreath determinant} and \eqref{eq:explicit Kronecker product},
it suffices to prove that
\begin{equation*}
\e_{r+1}(i,j)=M_{r+1}((k_{r+1}\floor{i/N_{r+1}}-\floor{j/M_{r+1}})\bmod m_{r+1})+\e_r(i\bmod N_{r+1},j\bmod M_{r+1})
\end{equation*}
for $r\ge1$.
For this purpose, we have only to see
\begin{equation*}
M_s((k_s\floor{i/N_s})-\floor{j/M_s})
\equiv
M_s((k_s\floor{(i\bmod N_{r+1})/N_s})-\floor{(j\bmod M_{r+1})/M_s})
\pmod{m_s}
\end{equation*}
when $s\le r$. This is easily verified since $n_s\mid(N_{r+1}/N_s)$ and $m_s\mid(M_{r+1}/M_s)$.
\end{proof}

By the lemmas above, we have
\begin{equation*}
\wM(G,H,\ords,\specp)
=TJDP(\tau)^{-1},
\end{equation*}
where
\begin{gather*}
T=P(\sigma_l)T(n_l;q^{k_lM_l})\otimes\dotsb\otimes P(\sigma_1)T(n_1;q^{k_1M_1}), \\
J=(I_{n_l}\otimes\1_{1,k_l})\otimes\dotsb\otimes(I_{n_1}\otimes\1_{1,k_1}), \\
D=(I_{n_l}\otimes\diag(q^{(k_l-1)M_l},\dots,q^{M_l},1))\otimes\dotsb\otimes
(I_{n_1}\otimes\diag(q^{(k_1-1)M_1},\dots,q^{M_1},1)),
\end{gather*}
and $\tau\in\sym m$ is determined by
\begin{gather*}
P(\tau)=P(\tau_l)\otimes\dotsb\otimes P(\tau_1)
\qquad(\tau_s=(1~2~\dots~m_s)\in\sym{m_s}).
\end{gather*}
We see that
\begin{equation*}
J=(I_n\otimes\1_{1,k})P(\sigma)
\end{equation*}
for some $\sigma\in\sym m$.
Thus we have
\begin{equation*}
\wM(G,H,\ords,\specp)=T\cdot(I_n\otimes\1_{1,k})P(\sigma\tau^{-1})\cdot P(\tau)DP(\tau)^{-1}.
\end{equation*}
It follows then
\begin{equation*}
\wrdet_k \wM(G,H,\ords,\specp)
=\omega^{(k^n)}(\sigma\tau^{-1})\Bigl(\frac{k!}{k^k}\Bigr)^n(\det T)^k\det D.
\end{equation*}
Since $\det P(\sigma_s)=(-1)^{n_s-1}$, we have
\begin{gather*}
(\det T)^k
=\prod_{s=1}^l \Bigl((q^{M_sm_s}-1)^{n_s-1}\Bigr)^{kn/n_s}
=\prod_{s=1}^l (q^{M_sm_s}-1)^{m(1-1/n_s)}, \\
\det D=\prod_{s=1}^l (q^{n_sM_sk_s(k_s-1)/2})^{m/m_s}
=\prod_{s=1}^l q^{mM_s(k_s-1)/2}.
\end{gather*}
This completes the proof of Theorem \ref{MainTheorem}. 

\section{Direct product case}
\subsection{Products of orderings and specializations}

Assume that
$G=G_1\times\dots\times G_l$,
and each $G_s$ is a group of order $m_s$ equipped with an ordering $\ord_s$
and a specialization $\spec_s\colon G_s\to R_s$, where $R_s$ is a ring.
Take a subgroup $H=H_1\times\dotsb\times H_l$ of $G$,
where $H_s$ is a subgroup of $G_s$ of order $n_s$ for each $s$.
We put $k_s=m_s/n_s$.
Let us fix a complete system of representatives $Z_s=\{z^s_0,z^s_1,\dots,z^s_{k_s-1}\}$ for each coset $G_s/H_s$.
We suppose that each ordering $\ord_s$ is a \emph{homogeneous ordering} in the sense that
\begin{equation*}
\ord_s(k_si+j)=z^s_j\;\ord_s(k_si)\quad(0\le j<k_s,\ 0\le i<n_s),
\qquad H_s=\set{\ord_s(k_si)}{0\le i<n_s}.
\end{equation*}

Put
\begin{equation}
M_j=\prod_{i=1}^{j-1}m_i\qquad(j=1,2,\dots,l),\qquad
m=m_1m_2\dots m_l=\card G,
\end{equation}
and take an ordering $\ord$ given by
\begin{equation}
g_{i}=\ord(i)
=\Bigl(\ord_1(\floor{i/M_1}\bmod m_1),\dots,\ord_l(\floor{i/M_l}\bmod m_l)\Bigr)
\qquad(i=0,1,\dots,m-1).
\end{equation}
We also take a specialization $\spec\colon G\to R=R_1\times\dots\times R_l$ given by
\begin{equation*}
\spec((x_1,\dots,x_l))=\spec_1(x_1)^{M_1}\dotsb\spec_l(x_l)^{M_l}\qquad(x_s\in G_s).
\end{equation*}

We have then
\begin{equation*}
\wM(G,H,\ord,\spec)=\left(
\prod_{s=1}^l \spec_s\bigl(\ord_s( (k_s\floor{i/N_s}-\floor{j/M_s})\bmod m_s)\bigr)^{M_s}
\right)_{\substack{0\le i<n \\ 0\le j<m}}.
\end{equation*}
By the same machinery in the discussion of the previous section,
we have
\begin{equation*}
\wM(G,H,\ord,\spec)=\wM(G_l,H_l,\ord_l,\spec_l^{M_l})\otimes\dotsb\otimes
\wM(G_2,H_2,\ord_2,\spec_2^{M_2})\otimes\wM(G_1,H_1,\ord_1,\spec_1^{M_1}),
\end{equation*}
where $\spec_s^{M_s}$ denotes the map which sends $g\in G_s$ to $\spec_s(g)^{M_s}\in R_s$.

\subsection{Special homogeneous case}

We look at the case where $l=1$.
We put
\begin{equation*}
H=\{h_0,h_1,\dots,h_{n-1}\},\qquad
Z=\{z_0,\dots,z_{k-1}\},
\end{equation*}
so that we have $G=\set{zh}{z\in Z,h\in H}$.
We choose $h_0=z_0$ to be the identity of $G$.
The homogeneous ordering of $G$ is
\begin{equation}
\ord(ik+j)=z_jh_i\qquad(0\le i< n,\ 0\le j< k).
\end{equation}
If we can factor the matrix $\wM(G,H,\ord,\spec)$ as
\begin{equation*}
\wM(G,H,\ord,\spec)=P(\sigma)\cdot\wM(H,\spec)\otimes\1_{1,k}
\cdot I_n\otimes\Psi(Z)\cdot P(\tau)^{-1},\qquad
\Psi(Z)=\diag(\psi(z_0),\psi(z_1),\dots,\psi(z_{k-1}))
\end{equation*}
for some $\sigma,\tau\in\sym m$ and some function $\psi\colon Z\to R$,
then we call the specialization $\spec$ to be \emph{separable} along with $\psi$.
If $\spec$ is separable, then we have
\begin{gather*}
\wD(G,H,\ord,\spec)
=(\sgn\sigma)^k\omega^{(k^n)}(\tau^{-1})
\Bigl(\frac{k!}{k^k}\Bigr)^n\prod_{s=0}^{k-1}\psi(z_s)^n
\wD(H,\spec)^k.
\end{gather*}

\begin{ex}
If $H=G$, then we have
\begin{equation*}
\wM(G,G,\ord,\spec)=P(e)\cdot\wM(G,\spec)\otimes\1_{1,1}
\cdot I_n\otimes\Psi\cdot P(e)^{-1},\qquad
\Psi=(1).
\end{equation*}
Hence $\spec$ is separable.
\end{ex}

\begin{ex}
If $H=\{e\}$, then we have
\begin{equation*}
\wM(G,\{e\},\ord,\spec)=P(e)\cdot\wM(\{e\},\spec)\otimes\1_{1,m}
\cdot I_1\otimes\Psi\cdot P(e)^{-1},\qquad
\Psi=\diag(\psi(g_0),\psi(g_1),\dots,\psi(g_{m-1})),
\end{equation*}
where $\psi(g)=\spec(g^{-1})/\spec(e)$.
Hence $\spec$ is separable.
\end{ex}

By the same discussion in the finite abelian case,
we have the
\begin{thm}
Let $G_1,\dots,G_l$ be finite groups, and $H_s$ be a subgroup of $G_s$ for each $s=1,\dots,l$.
Fix a complete system of representatives $Z_s$ for each coset $G_s/H_s$.
Denote by $\ord_s$, $\spec_s$ homogeneous orderings and specializations for $(G_s,H_s)$ respectively.
If each specialization $\spec_s$ is separable along with a function $\psi_s$,
then the wreath determinant for the pair $G=G_1\times\dotsb\times G_l$
and $H=H_1 \times\dotsb\times H_l$ is
\begin{equation*}
\wD(G,H)=(\sgn\sigma)^k
\omega^{(k^n)}(\sigma\tau^{-1})
\Bigl(\frac{k!}{k^k}\Bigr)^n
\prod_{s=1}^l\Psi(Z_s)^{m/k_s}
\prod_{s=1}^l\wD(H_s;\spec_s^{M_s})^{m/n_s},
\end{equation*}
where $m_s=\card{G_s}$, $n_s=\card{H_s}$, $k_s=\card{Z_s}$, $m=\card{G}$, $n=\card{H}$, $k=\card{G/H}$
and $\sigma,\tau$ are certain permutations of $m$ letters.
\qed
\end{thm}

\section{Further examples}

\subsection{Order specialization and Cayley specialization}

In the main part of the paper,
we exclusively discuss the wreath determinant
$\wD(G,H,\ord,\spec)$ with the principal specialization $\spec=\specp$
defined by $\specp(g_i)=q^{i}\in\C[q]$.
In what follows, we introduce two more kinds of specializations
and give several examples concerning such specializations.

\subsubsection{Order specialization}

We consider the specialization $\speco\colon G\to\C[q]$
defined by $\speco(g)=q^{o(g)-1}$,
where $o(g)$ is the order of $g$.

\begin{ex}
For $G=\cy 6=\{0,1,2,3,4,5\}$, we have
\begin{equation*}
\speco(0)=1,\quad \speco(1)=q^5,\quad \speco(2)=q^2,\quad
\speco(3)=q,\quad \speco(4)=q^2,\quad \speco(5)=q^5.
\end{equation*}
\end{ex}

\subsubsection{Cayley specialization}

Let $S$ be a symmetric generating set of $G$.
Then $(G,S)$ defines an undirected graph so called \emph{Cayley graph}.
For $x,y\in G$,
denote by $d(x,y)$ the Cayley distance between $x$ and $y$
(i.e. the length of the shortest path connecting $x$ and $y$ in the Cayley graph $(G,S)$).
We call the specialization $\specC\colon G\to\C[q]$ defined by
$\specC(g)=q^{d(g,e)}$ the \emph{Cayley specialization} with respect to $S$.

\begin{ex}
If $G=\sym n$ and $S=\{\text{transpositions}\}$, then
\begin{equation*}
d(g,e)=\nu(g)=n-\text{(number of cycles in $g$)}.
\end{equation*}
\end{ex}

\subsection{Dihedral groups}

Let
$G=D_m=\gen{\sigma,\tau\,\middle|\,\sigma^m=\tau^2=e,\ \sigma\tau=\tau\sigma^{-1}}$
be the dihedral group of degree $m$.
We set
\begin{equation}
\ords(im+j)=g_{im+j}=\tau^i\sigma^{j-1}\qquad(i=0,1,\ j=0,1,2,\dots,m-1),
\end{equation}
which we call the standard ordering of $D_m$.

\begin{ex}
When $G=D_3$, we have
\begin{equation*}
g_0=e,\quad g_1=\sigma,\quad g_2=\sigma^2,\quad
g_3=\tau,\quad g_4=\tau\sigma,\quad g_5=\tau\sigma^2.
\end{equation*}
\end{ex}

We give several examples of the wreath determinants $\wD(D_m,H,\ords,\specp)$.

\begin{ex}
Since
\begin{equation*}
\wM(D_m,\gen\tau,\ords,\specp)
=\Bigl(
\begin{pmatrix}
1 & q^m \\
q^m & 1
\end{pmatrix}
\otimes\1_{1,m}
\Bigr)\cdot (I_2\otimes\diag(1,q,\dots,q^{m-1}))\cdot P(\tau)
\end{equation*}
for a certain $\tau\in\sym m^2$, we have
\begin{equation*}
\wD(D_m,\gen\tau,\ords,\specp)
=\omega^{(m,m)}(\tau)\Bigl(\frac{m!}{m^m}\Bigr)^2 q^{m(m-1)}\det\begin{pmatrix}
1 & q^m \\
q^m & 1
\end{pmatrix}^m
=\Bigl(\frac{m!}{m^m}\Bigr)^2q^{m(m-1)}(1-q^{2m})^m.
\end{equation*}
\end{ex}

\begin{ex}
Suppose that $m$ is even, and write $m=2k$.
Since
\begin{equation*}
\wM(D_m,\gen{\sigma^k},\ords,\specp)
=\Bigl(
\begin{pmatrix}
1 & q^k \\
q^k & 1
\end{pmatrix}
\otimes\1_{1,m}
\Bigr)\cdot (I_2\otimes\diag(1,q,\dots,q^{k-1},q^{2k},q^{2k+1},\dots,q^{3k-1}))\cdot P(\tau)
\end{equation*}
for a certain $\tau\in\sym m^2$, we have
\begin{equation*}
\wD(D_m,\gen{\sigma^k},\ords,\specp)
=\omega^{(m,m)}(\tau)\Bigl(\frac{m!}{m^m}\Bigr)^2 q^{k(k-1)+k(5k-1)}\det\begin{pmatrix}
1 & q^k \\
q^k & 1
\end{pmatrix}^m
=\Bigl(\frac{m!}{m^m}\Bigr)^2q^{2k(3k-1)}(1-q^m)^m.
\end{equation*}
\end{ex}

\begin{rem}
Though the example above seems to suggest that
$\wM(D_m,\gen{\sigma^k},\ords,\specp)$ is calculated explicitly
when $m=kn$ for some positive integer $n$,
it may not be so simple.
For instance, we have
\begin{equation*}
\wM(D_6,\gen{\sigma^2},\ords,\specp)=-\frac3{2^{18}}q^{42}(1-q^2)^4(1-q^6)^4\,A,
\end{equation*}
where $A=3+12q^2+6q^4-44q^6-84q^8-44q^{10}+6q^{12}+12q^{14}+3q^{16}$.
\end{rem}

\begin{ex}
We have
\begin{align*}
\wD(D_2,\gen\sigma,\ords,\specp)&=-\frac1{2^3}q^4(1-q^2)^2, \\
\wD(D_3,\gen\sigma,\ords,\specp)&=\frac1{2^5}q^9(1-q^2)^2(1-q^3)^2(1+2q-4q^3-2q^4), \\
\wD(D_4,\gen\sigma,\ords,\specp)&=-\frac1{2^6}q^{16}(1-q^2)^2(1-q^4)^4(1-3q^2+q^4), \\
\wD(D_5,\gen\sigma,\ords,\specp)&=\frac1{2^9}q^{25}(1-q^2)^2(1-q^5)^6(1+2q-4q^2-10q^3+3q^4+20q^5+8q^6-4q^7-2q^8), \\
\wD(D_6,\gen\sigma,\ords,\specp)&=-\frac1{2^{11}}q^{36}(1-q^2)^2(1-q^6)^8(4-22q^2+39q^4-22q^6+4q^8).
\end{align*}
\end{ex}

\begin{ex}[Order specializations]
We have
\begin{align*}
\wD(D_2,\gen\tau,\ords,\speco)&=
\frac{1}{2^3} q^2(1-q)^2, \\
\wD(D_3,\gen\tau,\ords,\speco)&=
\frac{2^2}{3^5} q^4(1-q^2)^3, \\
\wD(D_4,\gen\tau,\ords,\speco)&=
-\frac{3}{2^{12}} q^6 (1-q)^2 (1-q^2)^2 \left(1+8 q+8 q^3+q^4\right), \\
\wD(D_5,\gen\tau,\ords,\speco)&=
\frac{2^63^2}{5^9} q^8 (1-q^2)^2 (1-q^6)^3 (1-3q^2+q^4), \\
\wD(D_6,\gen\tau,\ords,\speco)&=
-\frac{5}{2^63^9} q^{10} (1-q)^6\,A,
\end{align*}
where
\begin{multline*}
A=6+40 q+120 q^2+252 q^3+425 q^4+612 q^5+774 q^6+884 q^7+923 q^8 \\
{}+884 q^9+774 q^{10}+612 q^{11}+425 q^{12}+252 q^{13}+120 q^{14}+40 q^{15}+6 q^{16}.
\end{multline*}
\end{ex}

\begin{ex}
Let $H=\gen\sigma$ and $Z=\{e,\tau\}$.
\begin{equation*}
\ordh(2i+j)=\tau^j\sigma^i\qquad(i=0,1,\ j=0,1,\dots,n-1).
\end{equation*}
Then
\begin{equation*}
\ordh(0)=e,\quad \ordh(1)=\tau,\quad \ordh(2)=\sigma,\quad
\ordh(3)=\tau\sigma,\quad \ordh(4)=\sigma^2,\quad \ordh(5)=\tau\sigma^2,\quad\dots
\end{equation*}
In this case, we have
\begin{align*}
\wD(D_2,\gen\sigma,\ordh,\specp)&=\frac1{2^2}q^2(1-q^4)^2,\\
\wD(D_3,\gen\sigma,\ordh,\specp)&=\frac1{2^5}q^3(1-q^2)^2(1-q^6)^2(4 + 8 q^2 + 6 q^4 + 2 q^6 + q^8),\\
\wD(D_4,\gen\sigma,\ordh,\specp)&=\frac1{2^6}q^4(1-q^4)^2(1-q^8)^4(4+q^8),\\
\wD(D_5,\gen\sigma,\ordh,\specp)&=\frac1{2^9}q^5(1-q^2)^2(1-q^{10})^6
(16 + 32 q^2 + 8 q^4 - 16 q^6 + 14 q^{10} + 8 q^{12} + 2 q^{14} + q^{16}),\\
\wD(D_6,\gen\sigma,\ordh,\specp)&=\frac1{2^{10}}q^6(1-q^4)^2(1-q^{12})^8(16 - 16 q^4 + 12 q^8 - q^{12} + q^{16}).
\end{align*}
\end{ex}

\subsection{Symmetric groups}

Let $G=\sym n$ be the symmetric group of degree $n$.
We order the elements in $G$ lexicographically.

\begin{ex}
When $G=\sym3$, we have
\begin{equation*}
g_0=123,\quad g_1=132,\quad g_2=213,\quad g_3=231,\quad g_4=312,\quad g_5=321
\end{equation*}
in one-line notation.
\end{ex}

\begin{ex}[Group determinants for $\sym n$]
We have
\begin{align*}
\wD(\sym 2, \specC)&=1-q^2, \\
\wD(\sym 3, \specC)&=(1-q^2)^5 (1-4q^2), \\
\wD(\sym 4, \specC)&=(1-q^2)^{23} (1-4q^2)^{10} (1-9q^2), \\
\wD(\sym 5, \specC)&=(1-q^2)^{119} (1-4q^2)^{78} (1-9q^2)^{17} (1-16q^2).
\end{align*}
\end{ex}

\begin{ex}[Group determinants for $A_n$]
We have
\begin{align*}
\wD(A_3, \specC)&=(1-q^2)^2 (1+2q^2), \\
\wD(A_4, \specC)&=(1-q^2)^{11} (1+11q^2), \\
\wD(A_5, \specC)&=(1-q^2)^{59} (1-4q^2)^{18} (1+6q^2)^{16} (1+35q^2+24q^4).
\end{align*}
\end{ex}

\subsection{Cayley graph for group-subgroup pair}

\begin{ex}
Let $G=\cy{12}$, $H=\{0,3,6,9\}<G$ and $S=\{2,4,5,7,8\}\subset G$.
From the triplet $(G,H,S)$, we construct an undirected graph whose vertices are elements of $G$
in such a way that
$h$ and $h+s$ are connected by an edge if and only if
(i) $h\in H$ and $s\in S\setminus H$
or
(ii) $h\in H$ and $s\in S\cap H$ \cite{Cid}. See Figure \ref{fig:Cid}.
Then we have
\begin{equation*}
\wrdet_3\bigl(q^{d(h,g)}\bigr)_{h\in H, g\in G}=
\frac{2^7}{3^{11}}q^8(1-q^2)^6(1-q^4)^3,
\end{equation*}
where $d(g,h)$ denotes the distance between two vertices $g$ and $h$ on the graph constructed above.
Notice that the distance $d(g,h)$ depends only on the difference $h-g$ for any $h\in H$ and $g\in G$.
On the other hand, we have
\begin{equation*}
\wD(G,H,\specC)=
-\frac{2^4}{3^{11}}q^8(1-q)^8(1-q^2)(5+12q+25q^2+52q^3+43q^4+12q^5-q^6)
\end{equation*}
for the Cayley specialization $\specC$ for the Cayley graph $(G,S\cup(-S))$.
\begin{figure}[htbp]
\begin{center}
\includegraphics[scale=1]{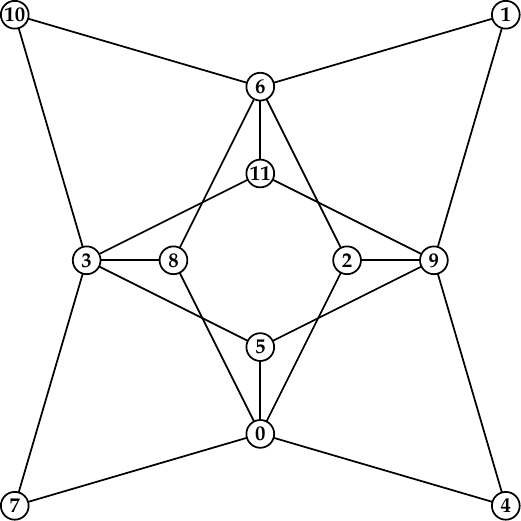}
\qquad
\qquad
\includegraphics[scale=1]{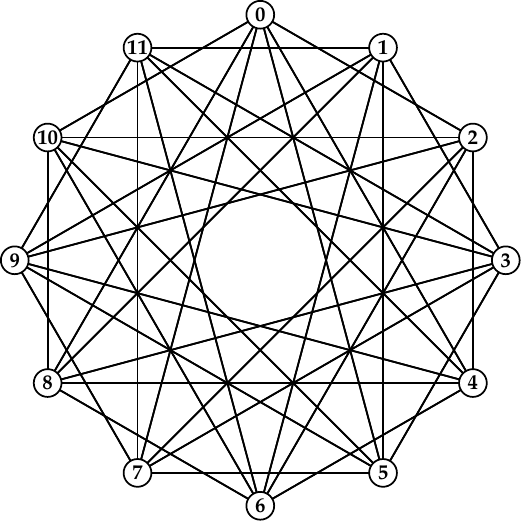}
\end{center}
\caption{Graphs for $(G,H,S)$ and $(G,S\cup(-S))$}\label{fig:Cid}
\end{figure}
\end{ex}

\begin{rem}
If one considers another specialization defined by $x_s=1$ for $s\in S$ and $x_g=0$ otherwise for a group-subgroup
pair, one may have a matrix corresponds to the rows associated with the elements of $H$
in the adjacency matrix of the group-subgroup pair graph, which is a generalization of the Cayley graphs.
See \S2 in \cite{Cid}.
\end{rem}

\begin{flushleft}
\bigskip

Kei Hamamoto \par
Graduate School of Mathematics, \par
Kyushu University \par
774 Motooka, Fukuoka 812-1234 JAPAN \par
\texttt{ma212034@math.kyushu-u.ac.jp}

\bigskip

Kazufumi Kimoto \par
Department of Mathematical Sciences, \par
University of the Ryukyus \par
1 Senbaru, Nishihara, Okinawa 903-0213 JAPAN \par
\texttt{kimoto@math.u-ryukyu.ac.jp}

\bigskip

Kazutoshi Tachibana \par
Shuyukan High School \par
6-1-10 Nishijin, Sawara-Ku, Fukuoka 814-8510 JAPAN \par
\texttt{tachibana-k3@fku.ed.jp}

\bigskip

Masato Wakayama \par
Institute of Mathematics for Industry, \par
Kyushu University \par
774 Motooka, Fukuoka 812-1234 JAPAN \par
\texttt{wakayama@imi.kyushu-u.ac.jp}
\end{flushleft}


\begin{thebibliography}{10}

\bibitem{AT1992}
N.~Alon and M.~Tarsi,
Coloring and orientations of graphs.
{\itshape Combinatorica} {\bfseries 12} (1992), 125--143.

\bibitem{C1999}
C.~W.~Curtis,
``Pioneers of Representation Theory: Frobenius, Burnside, Schur, and Brauer."
AMS \& LMS, History of Mathematics  {\bfseries 15}, 1999.

\bibitem{FS1991PAMS}
E.~Formanek and D.~Sibley,
The group determinant determines the group.
{\itshape Proc. Amer. Math. Soc.} {\bfseries 112} (1991), no. 3, 649--656.

\bibitem{K2014}
K.~Kimoto,
Averages of alpha-determinants over permutations.
Preprint, 2014. {\ttfamily arXiv:1403.3723}

\bibitem{KW2008}
K.~Kimoto and M.~Wakayama,
Invariant theory for singular $\alpha$-determinants.
{\itshape J.~Combin.~Theory Ser.~A} {\bfseries 115} (2008), no. 1, 1--31.

\bibitem{Mac}
I.~G.~Macdonald, ``Symmetric Functions and Hall Polynomials, Second Edition.''
Oxford Univ. Press, 1995.

\bibitem{MW2006}
S.~Matsumoto and M.~Wakayama,
Alpha-determinant cyclic modules of $\mathfrak{gl}_n(\mathbb{C})$.
{\itshape J.~Lie Theory} {\bfseries 16} (2006), no. 2, 393--405.

\bibitem{Cid}
C.~Reyes-Bustos,
Group-Subgroup Pair Graph.
Preprint, 2014.

\bibitem{ST2003}
T.~Shirai and Y.~Takahashi,
Random point fields associated with certain Fredholm determinants. I. Fermion, Poisson and boson point processes.
{\itshape J.~Funct.~Anal.} {\bfseries 205} (2003), no. 2, 414--463.

\bibitem{VJ1988}
D.~Vere-Jones, A generalization of permanents and determinants.
{\itshape Linear Algebra Appl.} {\bfseries 63} (1988), 267--270.
\end{thebibliography}
\end{document}